\documentclass[final,1p,times]{elsarticle}

\usepackage[colorlinks=true]{hyperref}
\usepackage{mathrsfs}

\usepackage{amsthm,amsmath,amssymb,mathrsfs,amsfonts,graphicx,graphics,latexsym,exscale,cmmib57,dsfont,bbm,amscd,ulem}
\usepackage{color,soul}
\newtheorem{thm}{Theorem}[section]
\newtheorem{prop}[thm]{Proposition}
\newtheorem{lem}[thm]{Lemma}
\newtheorem{cor}[thm]{Corollary}
\newtheorem{rem}[thm]{Remark}

\newcommand{\lam}[1]{\lambda_}

\newcommand{\bea}{\begin{eqnarray}}
\newcommand{\eea}{\end{eqnarray}}

\theoremstyle{definition}%

\newtheorem{asser}[thm]{Assertion}

\newcommand{\N}{\boldsymbol{N}}

\def\bA{{\boldsymbol{A}}}

\numberwithin{equation}{section}

\journal{Nonlinearity}
\begin{document}

\begin{frontmatter}
\title{Criterion for quasi-anosovian diffeomorphisms of closed manifolds}

\author{Xiongping Dai}
\ead{xpdai@nju.edu.cn}


\address{Department of Mathematics, Nanjing University, Nanjing 210093, People's Republic of China}

\begin{abstract}
Quasi-Anosov diffeomorphism is a kind of important dynamical system due to R.~Ma\~{n}\'{e} 1970s, although which is weaker than Anosov, yet has very stable dynamical behaviors---Axiom A and the no cycle condition. In this paper, we present a criterion for such dynamics using ergodic theory.
\end{abstract}

\end{frontmatter}

\section{Introduction}\label{sec1}%

Throughout this paper, let $M$ be a closed (i.e. smooth, compact, boundaryless) manifold of dimension $\dim M\ge2$. For a $\mathrm{C}^1$-diffeomorphism $f\colon M\rightarrow M$ of $M$, as usual $f$ is called \textit{Anosov} if there exists a (continuous) $T\!f$-invariant splitting of the tangent bundle $T\!{M}$ into subbundles
\begin{equation*}
T_xM=E^s(x)\oplus E^u(x)\quad \forall x\in M,
\end{equation*}
and if there are constants $C>0, \lambda>0$ such that
\begin{equation*}
\pmb{\|}T_xf^n|E^s(x)\pmb{\|}\le C e^{-n\lambda}\quad \textrm{and}\quad\pmb{\|}T_xf^n|E^u(x)\pmb{\|}_{\min}\ge C^{-1} e^{n\lambda}
\end{equation*}
for all $n\ge0$ and any $x\in M$.

We say that $f$ is \textit{quasi-Anosov} if for any $x\in M$ and any unit tangent vector $v\in T_xM$, the bi-sided infinite sequence $\{\|T_xf^n(v)\|\}_{-\infty<n<+\infty}$ is unbounded (cf.~Ma\~{n}\'{e}~\cite{M77}).

Although a quasi-Anosov diffeomorphism is strictly weaker than an Anosov system from the counterexample of Franks and Robinson~\cite{FR}, it still implies the very strong dynamical behaviors: Axiom A and the no cycle condition (cf.~Ma\~{n}\'{e}~\cite{M-LNM}).

By $\mathcal{M}_f$ it means the set of all $f$-invariant Borel probability measures on $M$. Based on the classical paper \cite{Atk} and the recent work~\cite{Dai11, Dai-pre}, we can now present a criterion for quasi-Anosov diffeomorphisms as follows.

\begin{thm}\label{thm1.1}
Let $T\!{M}$ have $T\!f$-invariant splitting $T_xM=E^s(x)\oplus E^{cu}(x)\, \forall x\in M$ such that $x\mapsto E^{cu}(x)$ is ``a.e.'' continuous and $x\mapsto E^s(x)$ is measurable. Assume
\begin{align}
\limsup_{n\to\infty}\|T_xf^n(v)\|<1&&\forall v\in E^s(x)\textrm{ and }\|v\|=1\label{eq1.1}
\intertext{and}
\liminf_{n\to\infty}\|T_xf^n(v)\|>1&&\forall v\in E^{cu}(x)\textrm{ and }\|v\|=1\label{eq1.2}
\end{align}
for ``a.e.'' $x\in M$ (where ``a.e.'' is in the sense of all $\mu\in\mathcal{M}_f$). Then $f$ is quasi-Anosov. And there exist two constants $K>0$ and $\varrho>1$ such that to any $x\in M$ there corresponds some $\ell(x)>0$ verifying that
\begin{align*}
&\|T_xf^{n+m}(v)\|\ge K\varrho^m\|T_xf^n(v)\|\quad\forall v\in E^{cu}(x)\textrm{ and }n\ge\ell(x),\\
\intertext{and}
&\|T_xf^{n-m}(v)\|\ge K\varrho^{-m}\|T_xf^n(v)\|\quad\forall v\in E^{s}(x)\textrm{ and }n\le-\ell(x),
\end{align*}
for all $m\ge0$.
\end{thm}
\begin{rem}
If instead $T_xM=E^s(x)\oplus E^{cu}(x)$ is such that $x\mapsto E^{s}(x)$ is ``a.e.'' continuous and $x\mapsto E^{cu}(x)$ is measurable, then the statements of Theorem~\ref{thm1.1} also hold.
\end{rem}

Our conditions (\ref{eq1.1}) and (\ref{eq1.2}) look much more weaker than the domination property of the splitting $E^s\oplus E^{cu}$. To prove this theorem, we shall first show that $f$ is nonuniformly hyperbolic for a.e. $x$ in $M$ using ergodic theory, and then using ergodic theory again we will extend the nonuniform hyperbolicity from a.e. $x\in M$ to every $x\in M$. This implies that $f$ is quasi-Anosov.

The remains of this paper will be simply organized as follows. We will introduce our main ergodic-theoretic tools in Section~\ref{sec2} and then we will prove Theorem~\ref{thm1.1} in Section~\ref{sec3}.

\section{Non-oscillatory behavior of a subadditive random process}\label{sec2}
To prove our Theorem~\ref{thm1.1}, we will need a result similar to Giles Atkinson's theorem on additive cocycles~\cite{Atk}. Atkinson's theorem (together with a result of K.~Schmidt) asserts the following.

\begin{lem}\label{lem2.1}
If $T\colon(X,\mathscr{F},\mu)\rightarrow(X,\mathscr{F},\mu)$ is an ergodic measure-preserving automorphism and $f\colon X\rightarrow\mathbb{R}$ is an integrable function with $\int_Xfd\mu=0$, then for $\mu$-a.e. $x\in X$ the sum $\sum_{k=0}^{n-1}f(T^kx)$ returns arbitrarily close to zero infinitely often.
\end{lem}

Atkinson's theorem has recently been extended for quasi-additive potentials~\cite{Dai-13}; and see \cite[Theorem~2.4]{Dai-jde} for a generalization for bounded subadditive process.

For a general, not necessarily bounded, subadditive process, the following similar lemma has not previously been formally published, but arose in discussion between Dr. Vaughn Climenhaga and Dr. Ian Morris on the MathOverflow internet forum, where their proof is adapted from G.~Atkinson's argument.\footnote{Cf.~http://mathoverflow.net/questions/70676/ for the details.}

\begin{lem}[Climenhaga and Morris]
Let $T$ be an ergodic measure-preserving transformation of a probability space $(X,\mathscr{F},\mu)$, and let $(f_n)_{n\ge1}$ be a sequence of integrable functions from $X$ to $\mathbb{R}$, which satisfies the subadditivity relation:
\begin{equation*}
f_{n+m}(x)\le f_n(T^mx)+f_m(x)\quad \textrm{for }\mu\textrm{-a.e. }x\in X\textrm{ and }n,m\ge1.
\end{equation*}
Suppose that ${\lim}_{n\to\infty}f_n(x)=-\infty$ for $\mu$-a.e. $x\in X$. Then
${\lim}_{n\to\infty}n^{-1}\int_Xf_n(x)d\mu(x)<0$.
\end{lem}

To prove our Theorem~\ref{thm1.1}, we shall need the following more general version proved in the recent paper~\cite[Theorem~2.7]{Dai-pre}.

\begin{lem}[\cite{Dai-pre}]\label{lem2.3}
Let $T$ be a measure-preserving, not necessarily ergodic, transformation of a probability space $(X,\mathscr{F},\mu)$, and let $(f_n)_{n\ge1}$ be a sequence of measurable functions from $X$ to $\mathbb{R}\cup\{-\infty\}$ with $f_1^+\in L^1(\mu)$, which satisfies the subadditivity relation:
\begin{equation*}
f_{n+m}(x)\le f_n(T^mx)+f_m(x)\quad \textrm{for }\mu\textrm{-a.e. }x\in X\textrm{ and }n,m\ge1.
\end{equation*}
Let $F(x)={\limsup}_{n\to\infty}f_n(x)$ for $x\in X$.
Then the symmetric difference
\begin{equation*}
\left\{x\in X\,|\,F(x)<0\right\}\vartriangle\left\{x\in X\,|\,{\lim}_{n\to\infty}n^{-1}f_n(x)<0\right\}
\end{equation*}
has $\mu$-measure $0$.
\end{lem}

Besides proving Theorem~\ref{thm1.1}, we here first present a simple application of Lemma~\ref{lem2.3}.
Given any metric system $T\colon(X,\mathscr{F},\mu)\rightarrow(X,\mathscr{F},\mu)$, for $E\in\mathscr{F}$ with $\mu(E)>0$, define the function
\begin{equation*}
1_E^{*}(x)=\liminf_{n\to\infty}\frac{1}{n}\#\{0\le i<n\,|\,T^ix\in E\}.
\end{equation*}
Then from Lemma~\ref{lem2.3}, we can easily obtain the following, which contains more than Poincar\'{e}'s recurrence theorem.

\begin{cor}\label{cor2.4}
Let $T$ be a measure-preserving, not necessarily ergodic, transformation of a probability space $(X,\mathscr{F},\mu)$.
Then for any $E\in\mathscr{F}$ with $\mu(E)>0$, it holds that $1_E^{*}(x)>0$ for $\mu$-a.e. $x\in E$.
\end{cor}

\begin{proof}
Since $f_n(x)=\sum_{i=0}^{n-1}1_E(T^ix)$ is an additive sequence, by Lemma~\ref{lem2.3} we have
\begin{equation*}
\left\{x\in X\colon\liminf_{n\to\infty}\sum_{i=0}^{n-1}1_E(T^ix)>0\right\}=\left\{x\in X\,|\,1_E^{*}(x)>0\right\}\quad \mu\textrm{-mod }0.
\end{equation*}
In addition, $x\in E$ implies $1_E(x)>0$ and hence $\liminf_{n\to\infty}\sum_{i=0}^{n-1}1_E(T^ix)\ge1$. This completes the proof of Corollary~\ref{cor2.4}.
\end{proof}

We note here that if $T\colon(X,\mathscr{F},\mu)\rightarrow(X,\mathscr{F},\mu)$ is ergodic or the ergodic decomposition theorem is applicable here, then Corollary~\ref{cor2.4} can be directly proved from the Birkhoff ergodic theorem as follows:

\begin{proof}
Let $T$ be \textit{\textbf{ergodic}} and $E\in\mathscr{F}$ with $\mu(E)>0$. Set $Z=\left\{x\in E\colon 1_E^*(x)=0\right\}$ and on the contrary let $\mu(Z)>0$. By the Birkhoff ergodic theorem~\cite{Wal82}, $\frac{1}{n}\sum_{i=0}^{n-1}1_Z(T^ix)$ converges $\mu$-a.e. and in $L^1(\mu)$-norm to $1_Z^*(x)$. Since $Z\subseteq E$ and then $0\le 1_Z(x)\le 1_E(x)$, there follows
$1_Z^*(x)=0$ for $\mu$-a.e. $x\in Z$. Because $1_Z^*(x)$ is $T$-invariant and $\mu$ is ergodic, we see $1_Z^*(x)\equiv 0$ $\mu$-a.e. and further
\begin{equation*}
0=\int_X1_Z^*(x)d\mu(x)=\lim_{n\to\infty}\frac{1}{n}\sum_{i=0}^{n-1}\int_X1_Z(T^ix)d\mu(x)=\mu(Z).
\end{equation*}
This contradiction completes the proof of Corollary~\ref{cor2.4} in the ergodic case.
\end{proof}

In the above proof, the ergodicity of $\mu$ plays a role to guarantee $1_Z^*(x)\equiv 0$. However, in the situation of Corollary~\ref{cor2.4}, the classical ergodic decomposition is not applicable, since $(X,\mathscr{F})$ is not necessarily to be a Borel space.

As a consequence of Lemma~\ref{lem2.3} and a theorem of Froyland, LLoyd and Quas~\cite[Theorem~4.1]{FLQ}, we have obtained the following multiplicative ergodic theorem \cite[Theorem~1.1]{Dai-pre}.

\begin{lem}[\cite{Dai-pre}]\label{lem2.5}
Let $T$ be a measure-preserving transformation of a Polish probability space $(X,\mathscr{F},\mu)$ and assume $\bA\colon\mathbb{Z}_+\times X\rightarrow\mathbb{R}^{d\times d}$ is a measurable cocycle driven by $T$ such that $\log^+\pmb{\|}\bA(1,\cdot)\pmb{\|}\in L^1(\mu)$. Then there exists a set $B\in\mathscr{F}$ with $T(B)\subseteq B$ and $\mu(B)=1$ such that:
\begin{enumerate}
\item[$(\mathrm{a})$] There is a measurable function $s\colon B\rightarrow\mathbb{N}$ with $s\circ T=s$.

\item[$(\mathrm{b})$] If $x$ belongs to $B$ there are $s(x)$ numbers $-\infty=\lambda_1(x)<\lambda_2(x)<\dotsm<\lambda_{s(x)}(x)<\infty$.

\item[$(\mathrm{c})$] There is a measurable filtration of $\mathbb{R}^d$:
\begin{equation*}
\varnothing=V^{(0)}(x)\subset V^{(1)}(x)\subset\dotsm\subset V^{(s(x))}(x)=\mathbb{R}^d\quad \forall x\in B.
\end{equation*}

\item[$(\mathrm{d})$] If $x$ belongs to $B$, then
\begin{enumerate}
\item[$(\mathrm{i})$] for $1\le i\le s(x)$,
\begin{equation*}
\lim_{n\to\infty}\frac{1}{n}\log\|\bA(n,x)v\|=\lambda_i(x)\quad \forall v\in V^{(i)}(x)\setminus V^{(i-1)}(x);
\end{equation*}
\item[$(\mathrm{ii})$] for $2\le i\le s(x)$, one can find some $v_i\in V^{(i)}(x)\setminus V^{(i-1)}(x)$ such that
\begin{equation*}
\limsup_{n\to\infty}e^{-\lambda_i(x)n}\|\bA(n,x)v_i\|\ge\|v_i\|.
\end{equation*}
\end{enumerate}

\item[$(\mathrm{e})$] The function $\lambda_i(x)$ is defined and measurable on $\{x\,|\,s(x)\ge i\}$ and $\lambda_i^{}(T(x))=\lambda_i(x)$ on this set.

\item[$(\mathrm{f})$] For any $x\in B$ and all $1\le i\le s(x)$,
\begin{enumerate}
\item[$(\mathrm{i})$] $\bA(1,x)V^{(i)}(x)\subseteq V^{(i)}(T(x))$ and
\item[$(\mathrm{ii})$] $\dim V^{(i)}(T(x))=\dim V^{(i)}(x)$.
\end{enumerate}
\end{enumerate}
\end{lem}

Here the new main point of our MET that we will need later is the property (d)-(ii).

\section{Expanding cocycles and quasi-Anosov diffeomorphisms}\label{sec3}
In this section, together with another ergodic-theoretic tool (Lemma~\ref{lem3.1} below), we will present applications of Lemmas~\ref{lem2.1}, \ref{lem2.3} and \ref{lem2.5} to differentiable dynamical systems including proving our main result Theorem~\ref{thm1.1}.
\subsection{Expanding cocycles}\label{sec3.1}
Let $T\colon X\rightarrow X$ be a continuous transformation of a compact metric space $X$. We denote by $\mathcal{M}_T$ the space of all $T$-invariant Borel probability measures on $X$. Given a Borel measurable function $\varphi\colon X\rightarrow\mathbb{R}\cup\{\pm\infty\}$, it is said to be \textit{a.e. continuous with respect to $T$} if the set $D_\varphi$ of discontinuities of $\varphi$ is such that $\mu(D_\varphi)=0$ for each $\mu\in\mathcal{M}_T$.

Besides the lemmas sated in Section~\ref{sec2}, we will need the following semi-uniform subadditive ergodic theorem.

\begin{lem}[{\cite[Theorem~3.3]{Dai11}}]\label{lem3.1}
Let $T\colon X\rightarrow X$ be a continuous transformation of the compact metric space $X$, and $\varphi_n\colon X\rightarrow\mathbb{R}\cup\{-\infty\}$ a subadditive sequence of a.e. continuous upper-bounded functions with respect to $T$. If the set
\begin{equation*}
\Gamma:=\left\{x\in X\colon\limsup_{n\to\infty}\frac{1}{n}\varphi_n(x)<0\right\}
\end{equation*}
is of total measure $1$, i.e. $\mu(\Gamma)=1$ for all $\mu\in\mathcal{M}_T$, then $\Gamma=X$.
\end{lem}

We note that readers can also see \cite{Sch, SS, Cao} for the semi-uniform subadditive ergodic theorem in the case that $\varphi_n(x)\in\mathbb{R}$ are continuous in $x$ for all $n\ge1$.

Let $A\colon X\rightarrow\mathbb{R}^{d\times d}$ be a measurable matrix-valued map, where $1\le d<\infty$ is an integer. Then
\begin{equation}
\bA\colon\mathbb{Z}_+\times X\rightarrow\mathbb{R}^{d\times d};\quad (n,x)\mapsto\begin{cases}A(T^{n-1}(x))\dotsm A(x)& \textrm{if }n\ge1,\\ I_{d\times d}& \textrm{if } n=0,\end{cases}
\end{equation}
is a measurable cocycle driven by $T$.

Then the following result is a consequence of Lemma~\ref{lem3.1} and Lemma~\ref{lem2.5}.

\begin{prop}\label{prop3.2}
Let $T\colon X\rightarrow X$ be a continuous transformation of the compact metric space $X$ and $A\colon X\rightarrow\mathbb{R}^{d\times d}$ an a.e. continuous matrix-valued map with $\sup_{x\in X}\pmb{\|}A(x)\pmb{\|}<\infty$. If the set of ``quasi-stable points'' of $\bA$
\begin{equation*}
\varLambda=\left\{x\in X\,\big{|}\,\limsup_{n\to\infty}\|\bA(n,x)v\|<\|v\|\;\forall v\in\mathbb{R}^d\setminus\{0\}\right\}
\end{equation*}
has the total measure $1$, then $\limsup_{n\to\infty}\frac{1}{n}\log\pmb{\|}\bA(n,x)\pmb{\|}<0$ for each $x\in X$.
\end{prop}

\begin{proof}
Clearly, by the assumption every Lyapunov exponents of $\bA$ are nonpositive. Further from the improved multiplicative ergodic theorem (Lemma~\ref{lem2.5}), it follows that every Lyapunov exponents of $\bA$ are negative.

Let $f_n(x)=\log\pmb{\|}\bA(n,x)\pmb{\|}$. Then $f_n$ is subadditive and a.e. continuous with respect to $T$.
Thus this statement follows immediately from Lemma~\ref{lem3.1}.
\end{proof}

We say $A(x)=(a_{ij}(x))\in\mathbb{R}^{d\times d}$ is upper-semi continuous if every elements $a_{ij}(x)$ are upper-semi continuous with respect to $x$, for $1\le i,j\le d$. The lower-semi continuity of $A(x)$ may be similarly defined.

\begin{cor}\label{cor3.3}
Let $T\colon X\rightarrow X$ be a continuous transformation of the compact metric space $X$ and assume $A\colon X\rightarrow\mathbb{R}^{d\times d}$ is an a.e. and upper-semi continuous matrix-valued map. If the set of quasi-stable points $\varLambda$ of $\bA$ has the total measure $1$, then $\bA$ is uniformly contractive on $X$.
\end{cor}

\begin{proof}
From Proposition~\ref{prop3.2}, it follows that for any $x\in X$, $\limsup_nn^{-1}\log\pmb{\|}\bA(n,x)\pmb{\|}<0$. By the upper-semi continuity of $\pmb{\|}\bA(n,x)\pmb{\|}$, one can find an open neighborhood $U(x)$ and an integer $L(x)\ge1$ such that $\pmb{\|}\bA(L(x),y)\pmb{\|}\le\lambda_x\,(<1)$ for all $y\in U(x)$. Then by choosing an open cover $U(x_1), \dotsc,U(x_k)$ of $X$, we can pick constants $C>0$ and $0<\lambda<1$ such that
\begin{equation*}
\pmb{\|}\bA(n,x)\pmb{\|}\le C\lambda^n\quad \forall x\in X\textrm{ and }n\ge0.
\end{equation*}
This completes the proof of Corollary~\ref{cor3.3}.
\end{proof}

Since $\prod_{i=0}^{n-1}\pmb{\|}A(T^ix)\pmb{\|}_{\min}\le\|\bA(n,x)v\|$ for any unit vector $v\in\mathbb{R}^d$, our quasi-expanding condition in the following Corollary~\ref{cor3.4} is much more weaker than the usual nonuniformly expanding assumption (cf.~\cite{AAS,Cao03,COP,ST,Dai13}) that requires that $$\liminf_{m\to\infty}\frac{1}{m}\sum_{i=0}^{m-1}\log\pmb{\|}A(T^ix)\pmb{\|}_{\min}\ge\lambda\quad \textrm{a.e.}$$
for all $\mu\in\mathcal {M}_T$, for some universal constant $\lambda>0$.

\begin{cor}\label{cor3.4}
Let $T\colon X\rightarrow X$ be a continuous transformation of the compact metric space $X$ and suppose that $A\colon X\rightarrow\mathbb{R}^{d\times d}$ is an a.e and lower-semi continuous nonsingular-matrix-valued map. If the set of ``quasi-expanding points'' of $\bA$
\begin{equation*}
\varDelta=\left\{x\in X\big{|}\liminf_{m\to\infty}\|\bA(m,x)v\|>\|v\|\;\forall v\in \mathbb{R}^d\setminus\{0\}\right\}
\end{equation*}
has total measure $1$, then $\bA$ is uniformly expanding on $X$.
\end{cor}

\begin{proof}
This statement comes from an argument similar to that of Corollary~\ref{cor3.3}.
\end{proof}

We note that if we additionally assume $A(x)$ is continuous on $X$ instead of the a.e. continuity, then Corollary~\ref{cor3.4} can be easily obtained by a simple modification of the argument of \cite{CLR}.

\subsection{Quasi-Anosov diffeomorphisms}\label{sec3.2}
Now it is time to prove Theorem~\ref{thm1.1} stated in Section~\ref{sec1}. From now on, let $f\colon M\rightarrow M$ be a $\mathrm{C}^1$-diffeomorphism of a closed manifold $M$, where the dimension $\dim M$ of $M$ is larger than or equal to $2$.

\begin{proof}[Proof of Theorem~\ref{thm1.1}]
Since $f$ is of $\mathrm{C}^1$-class, $\inf_{x\in M}\pmb{\|}T_xf|E^{cu}(x)\pmb{\|}_{\min}\ge\inf_{x\in M}\pmb{\|}T_xf\pmb{\|}_{\min}>0$. By Proposition~\ref{prop3.2} and condition (\ref{eq1.2}), it follows that $T\!f$ is (not necessarily uniformly) exponentially expanding restricted to $E^{cu}(x)$ for all $x\in M$.

In addition, by considering $f^{-1}$ instead of $f$, from Oselede\v{c}'s multiplicative ergodic theorem~\cite{Ose} and Proposition~\ref{prop3.2} we can see that $f^{-1}$ is nonuniformly contracting restricted to $E^{cu}(x)$ for every $x\in M$.

We next consider the $E^{s}$-subbundle.

\begin{asser}\label{asser3.5}
For every $\mu\in\mathcal{M}_f$, $f$ is nonuniformly contracting restricted to the $s$-subbundle $E^{s}(x)$ for $\mu$-a.e. $x\in M$.
\end{asser}

\begin{proof}
Let $E_x^{s,\sharp}=\{v\in E^{s}(x)\colon \|v\|=1\}$ and $E^{s,\sharp}=\bigsqcup_{x\in M}E_x^{s,\sharp}$ be the $s$-subbundle of unit tangent vectors to $M$. We define a natural measurable dynamical system
$$
F\colon E^{s,\sharp}\rightarrow E^{s,\sharp};\quad F(x,v)=\frac{T_xf(v)}{\|T_xf(v)\|}\; \forall x\in M\textrm{ and }v\in E_x^{s,\sharp},
$$
Since $f$ is $\mathrm{C}^1$-diffeomorphic, $F$ is well defined and continuous.
We introduce the standard potential function
$$
\varphi\colon E^{s,\sharp}\rightarrow\mathbb{R};\quad (x,v)\mapsto \varphi(v)=\log\|T_xf(v)\|\; \forall(x,v)\in E^{s,\sharp}.
$$
Since $f$ is of $\mathrm{C}^1$-class, it is easy to see that $|\varphi|$ is bounded on $E^{s,\sharp}$ such that
\begin{equation*}
\log\|T_xf^n(v)\|=\sum_{i=0}^{n-1}\varphi(F^i(x,v))\quad \forall (x,v)\in E^{s,\sharp}\textrm{ and }n\ge1.
\end{equation*}
So from the hypothesis (\ref{eq1.1}) of the theorem, it follows that
$$
\limsup_{n\to\infty}\sum_{i=0}^{n-1}\varphi(F^i(x,v))<0\quad \forall v\in E_x^{s,\sharp}\textrm{ and }\textrm{a.e. }x\in M.
$$
By $T^\sharp M$ we denote the bundle of unit tangent vectors to $M$ and we naturally extend $F$ to it.
Let $\pi\colon T^\sharp M\rightarrow M$ be the natural projection. Then $\pi_*\colon\mathcal{M}_F\rightarrow\mathcal{M}_f$ is surjective, where $\mathcal{M}_F$ stands for the set of all $F$-invariant Borel probability measures on $T^\sharp M$.

Thus for any given $F$-invariant ergodic Borel probability measure $\hat{\mu}$ on $T^\sharp M$, from Lemma~\ref{lem2.1} or Lemma~\ref{lem2.3}, it follows that
\begin{equation*}\begin{split}
\lambda(\hat{\mu})&:=\lim_{n\to\infty}\frac{1}{n}\log\|T_xf^n(v)\|\quad \hat{\mu}\textrm{-a.e. }(x,v)\in T^\#M\\
&\;<0 \qquad\textrm{if }\quad\hat{\mu}(E^{s,\sharp})>0.
\end{split}\end{equation*}
Given any ergodic $\mu\in\mathcal{M}_f$, from \cite{Dai09} follows that one can find $\dim M$ ergodic measures, say $\hat{\mu}_1,\dotsc,\hat{\mu}_{\dim M}\in\mathcal{M}_F$, such that
$\pi_*(\hat{\mu}_i)=\mu$ for $1\le i\le\dim M$ and $\{\lambda(\hat{\mu}_1),\dotsc,\lambda(\hat{\mu}_{\dim M})\}$ is just the Lyapunov characteristic spectrum of $f$ at $\mu$ counting with multiplicities. Thus
$$
\lim_{n\to\infty}\frac{1}{n}\log\|T_xf^n(v)\|<0\quad \forall v\in E_x^{s,\sharp},\quad \mu\textrm{-a.e. }x\in M,
$$
for any $\mu\in\mathcal{M}_f$. This completes the proof of Assertion~\ref{asser3.5}.
\end{proof}

By $\widehat{E}^{s}(x)$, we denote the orthogonal complement of $E^{cu}(x)$ in $T_xM$, for every $x\in M$. Since $E^{cu}(x)$ is a.e. continuous with respect to $x\in M$, so is $\widehat{E}^{s}(x)$. Let $\widehat{F}_x^{s}\colon\widehat{E}^{s}(x)\rightarrow\widehat{E}^{s}(f(x))$ be the linear isomorphism, which is naturally induced by the orthogonal projection, for each $x\in M$. According to Liao's version of multiplicative ergodic theorem, $\widehat{F}^{s}\colon\widehat{E}^{s}\rightarrow\widehat{E}^{s}$ is non-uniformly contractive for each $\mu\in\mathcal{M}_f$. Finally, Lemma~\ref{lem3.1} follows that $\widehat{F}^{s}$ is (not necessarily uniformly) contractive at each $x\in M$. Then for any $x\in M$, there exists two constants $C_x>0$ and $\lambda_x>0$ such that
\begin{equation*}
\pmb{\|}T_xf^n|E^s(x)\pmb{\|}\le C_xe^{-n\lambda_x}\quad \textrm{and}\quad\pmb{\|}\widehat{F}^{cu}_{f^{n-1}x}\circ\dotsm\circ\widehat{F}^{cu}_x|\widehat{E}^{cu}(x)\pmb{\|}_{\min}\ge C_x^{-1} e^{n\lambda_x}
\end{equation*}
for all $n\ge0$. 

Moreover by considering $f^{-1}$ instead of $f$, we can see that $f^{-1}$ is nonuniformly expanding restricted to $\widehat{E}^s(x)$ for every $x\in M$.

Therefore, for any $x\in M$ and any $u\in T_xM-(\widehat{E}^s(x)\cup E^{cu}(x))$ we have $u=u_-+u_+$ where $u_-\in \widehat{E}^s(x)$ and $u_+\in\widehat{E}^{cu}(x)\setminus\{0\}$ such that
\begin{equation*}
\|T_xf^n(u)\|=\|T_xf^n(u_+)+T_xf^n(u_-)\|\to+\infty\quad \textrm{as }n\to\infty.
\end{equation*}
Thus $f$ is quasi-Anosov.

Next we proceed to prove the second part of the theorem. By the quasi-hyperbolicity, we have the following.

\begin{asser}[\cite{M77}]\label{asser3.6}
There exists a positive integer $\N$ such that for any $v\in T_xM$ there is some $n\in[-\N,\N]$ with the property $\|T_xf^n(v)\|\ge2\|v\|$.
\end{asser}

From now on, let $\N$ be given as in Assertion~\ref{asser3.6}. We call $u(\not=0)\in T_xM$ is a \textit{right $\N$-altitude} if $\|u\|\ge\|T_xf^m(u)\|$ for all $-\N\le m\le 0$; and $u$ is called a \textit{left $\N$-altitude} if $\|u\|\ge\|T_xf^m(u)\|$ for all $0\le m\le \N$.

The above Assertion~\ref{asser3.6} immediately implies the following.

\begin{asser}\label{asser3.7}
Given any $x\in M$ and any nonzero $u\in T_xM$, for any integer $k$, there exists at least one right or left $\N$-altitude in $\left\{T_xf^{k+n}(u)\,|\,-\N\le n\le \N\right\}$.
\end{asser}

From definition and Assertion~\ref{asser3.6}, we can easily obtain the following.

\begin{asser}[{\cite[Lemma~1.3]{WG}}]\label{asser3.8}
Given any $x\in M$, if $u\in T_xM$ is a right $\N$-altitude, then for any $n\ge0$ and $m\ge0$, we have
\begin{equation*}
\|T_xf^{n+m}(u)\|\ge C^{-1}\lambda^{-m}\|T_xf^n(u)\|.
\end{equation*}
If $u\in T_xM$ is a left $\N$-altitude, then for any $n\ge0$ and $m\ge0$, we have
\begin{equation*}
\|T_xf^{-n-m}(u)\|\ge C^{-1}\lambda^{-m}\|T_xf^{-n}(u)\|.
\end{equation*}
Here
\begin{equation*}
\lambda=\left(\frac{1}{2}\right)^{\N}\quad \textrm{and}\quad C=\left(\frac{1}{2}\min\left\{\pmb{\|}T_xf^k\pmb{\|}_{\min}\colon x\in M,\;0\le k\le \N\right\}\right)^{-2}.
\end{equation*}
\end{asser}

Now let $x\in M$ be arbitrarily given. Since $\|T_xf^n(u)\|$ tends to $+\infty$ as $n\to\infty$ for any nonzero $u\in E^{cu}(x)$, we can find some integer $k(x,u)>0$ such that $T_xf^{k(x,u)}(u)$ is a right $\N$-altitude. Furthermore, since $T_xf(u)$ is continuous with respect to $x$ and $u$, we can choose some large integer $\ell(x)>0$ such that for each nonzero $u\in E^{cu}(x)$, there exists at least one right $\N$-altitude in $\left\{T_xf^n(u)\,|\,0\le n\le\ell(x)\right\}$. This together with Assertion~\ref{asser3.8} follows the second part of the theorem for the subbundle $E^{cu}$. Similarly we can prove the inequality for the subbundle $E^s$.

This thus completes the proof of Theorem~\ref{thm1.1}.
\end{proof}

\subsection{Anosov diffeomorphisms}\label{sec3.3}
We will need the following lemma which gives us a characterization of Anosov diffeomorphisms.

\begin{lem}[{\cite[Corollary~1]{M77}}]\label{lem3.9}
$f$ is Anosov if and only if it is quasi-Anosov and each periodic points of $f$ have the same stable index.
\end{lem}

The following result provides us a criterion for Anosov system.

\begin{thm}\label{thm3.10}
Let $T\!{M}$ have $T\!f$-invariant splitting $T_xM=E^s(x)\oplus E^{cu}(x)\, \forall x\in M$ such that $x\mapsto E^s(x)$ is continuous and $x\mapsto E^{cu}(x)$ is measurable with respect to $x\in M$. Assume
\begin{align*}
\limsup_{m\to\infty}\|T_xf^m(v)\|<1&&\forall v\in E^s(x),\ \|v\|=1
\intertext{and}
\liminf_{m\to\infty}\|T_xf^m(v)\|>1&&\forall v\in E^{cu}(x),\ \|v\|=1
\end{align*}
for ``a.e.'' $x\in M$. Then $f$ is Anosov.
\end{thm}

\begin{proof}
From Theorem~\ref{thm1.1}, it follows easily that $f$ is quasi-Anosov and that for any $x\in\Omega(f)$, $T_xW^s(x)=E^s(x)$. Since $E^s(x)$ is continuous with respect to $x\in M$ and $M$ is connected, $\dim E^s(x)$ is constant. Therefore $f$ is Anosov from Lemma~\ref{lem3.9}.

This completes the proof of Theorem~\ref{thm3.10}.
\end{proof}


\subsection*{\textbf{Acknowledgments}}%
The author would like to thank professor Shaobo Gan for him to point out the Lemma~\ref{lem3.9}.
This work was supported in part by National Science Foundation of China Grant $\#$11271183 and PAPD of Jiangsu Higher Education Institutions

\end{document}